\documentclass{amsart}
\usepackage[english]{babel}
\usepackage{amsmath,amscd,amssymb,amsthm,amsxtra}
\usepackage{microtype}
\usepackage{mathrsfs}
\usepackage{color,graphicx}
\usepackage{latexsym}
\usepackage{epsfig,epstopdf}
\usepackage[margin=1.5in]{geometry}

\theoremstyle{plain}
\newtheorem{thm}{Theorem}[section]

\newtheorem{lem}[thm]{Lemma}

\theoremstyle{definition}

\newtheorem{rmk}[thm]{Remark}
\numberwithin{thm}{section}
\numberwithin{equation}{section}
\numberwithin{figure}{section}

\def\1{{\bf 1}}
\def\l{\lambda}
\def\L{\Lambda}

\DeclareMathOperator{\Id}{Id}

\DeclareMathOperator*{\osc}{osc}
\DeclareMathOperator{\trace}{tr}

\begin{document}

\title[Nonlinear Dependence]{Nonlinear Bounds in H\"older Spaces for the Monge-Amp\`{e}re Equation}

\author[A.\ Figalli]{Alessio Figalli}
\address{The University of Texas at Austin, 
Mathematics Dept.,
Austin, TX 78712, USA}
\email{figalli@math.utexas.edu}

\author[Y.\ Jhaveri]{Yash Jhaveri}
\address{The University of Texas at Austin, 
Mathematics Dept.,
Austin, TX 78712, USA}
\email{yjhaveri@math.utexas.edu}

\author[C.\ Mooney]{Connor Mooney}
\address{The University of Texas at Austin, 
Mathematics Dept.,
Austin, TX 78712, USA}
\email{cmooney@math.utexas.edu}


\begin{abstract}
We demonstrate that $C^{2,\alpha}$ estimates for the Monge-Amp\`{e}re equation depend in a highly nonlinear way both on the $C^{\alpha}$ norm of the right-hand side and $1/\alpha$.
First, we show that if a solution is strictly convex, then the $C^{2,\alpha}$ norm of the solution depends polynomially on the $C^{\alpha}$ norm of the right-hand side. 
Second, we show that the $C^{2,\alpha}$ norm of the solution is controlled by $\exp((C/\alpha)\log(1/\alpha))$ as $\alpha \to 0$.
Finally, we construct a family of solutions in two dimensions to show the sharpness of our results.

\vspace{3mm}

\end{abstract}
\maketitle


\section{Introduction}
\label{sec:Intro}
The model case in the study of linear elliptic equations is the Poisson equation $\Delta v=g.$ 
It is well known that for such an equation, the solution $v$ gains two derivatives with respect to $g$ in H\"older spaces. 
More precisely, if
$$
\Delta v=g\qquad \text{in }B_1\subset \mathbb R^n,
$$
then there exists a constant $C > 0$ depending only on $n$ and $\alpha$ such that
\[ \|v\|_{C^{2,\alpha}(B_{1/2})} \leq C\big(\|v\|_{L^{\infty}(B_1)}  + \|g\|_{C^{\alpha}(B_1)}\big) \qquad \forall \alpha \in (0,1). \]
In addition, it is known that $C \sim 1/\alpha$ as $\alpha \to 0$.
This estimate extends easily to linear uniformly elliptic equations,
and even to fully nonlinear uniformly elliptic equations (see the Appendix).
It is therefore natural to ask whether or not one may extend such a well-quantified estimate to other more general elliptic equations.

In this note, we address this issue for the Monge-Amp\`{e}re equation investigating the following general question:\\ \\
Let $u: \mathbb{R}^n\to \mathbb{R}$ be a strictly convex solution to the Monge-Amp\`ere equation
\begin{equation}
\label{eqn:MA}
\det D^2u = f \qquad \text{in } \mathbb{R}^n,
\end{equation}
where $0 < \lambda \leq f \leq \Lambda < \infty$.
Is it possible to find a quantitative dependence for $\|u\|_{C^{2,\alpha}(B_{1/2})}$ in terms of $\|f\|_{C^{\alpha}(B_1)}$ and $\alpha$?
More specifically, among all solutions which have a fixed modulus of strict convexity, we look for a function $\omega:\mathbb{R}^+\times \mathbb{R}^+\to \mathbb{R}^+$, depending only on the dimension $n$, such that
the estimate
\begin{equation}
\label{eqn:omega}
\|u\|_{C^{2,\alpha}(B_{1/2})} \leq \omega\big(\alpha, \|f\|_{C^{\alpha}(B_1)}\big)
\end{equation}
holds. \\

Notice that by Caffarelli's $C^{2,\alpha}$ interior estimates for Monge-Amp\`ere (see \cite{C3}) and a compactness argument, one can show that such a function $\omega$ exists.
However, as we have already said, it would be interesting to understand 
how $\omega$ depends on its arguments.
Here we shall see that $\omega$ grows polynomially in $\|f\|_{C^{\alpha}(B_1)}$ and at least exponentially in $1/\alpha$ as $\alpha\to 0$,
which is in strong contrast with linear uniformly elliptic equations where $\omega$ depends linearly both on $\|f\|_{C^{\alpha}(B_1)}$ and $1/\alpha$ as $\alpha \to 0$.
In particular, we first prove the following two results:

\begin{thm}
\label{thm:positive1}
Let $u$ be a strictly convex solution to
\[ \det D^2u = f \qquad \text{in } B_1 \subset \mathbb{R}^n \]
with $0 < \lambda \leq f \leq \Lambda < \infty$ and $f \in C^{\alpha}(B_1)$. 
Then, there exist two constants $C>0$ and $\rho>1$ such that
\begin{equation*}
\label{eqn:poly}
\|D^2u\|_{C^{\alpha}(B_{1/2})} \leq C\|f\|_{C^{\alpha}(B_{1})}^\rho.
\end{equation*}
Here $C$ depends only on $\lambda,\Lambda,n,\alpha$, and the modulus of strict convexity of $u$, while $\rho $ depends only on $\lambda,\Lambda,n$, and $\alpha$.
\end{thm}

\begin{thm}
\label{thm:positive2}
Let $u$ be a strictly convex solution to
\[ \det D^2u = f \qquad \text{in } B_1 \subset \mathbb{R}^n \]
with $0 < \lambda \leq f \leq \Lambda < \infty$ and $f \in C^{\alpha}(B_1)$. 
Also, assume that $\|f\|_{C^{\alpha}(B_1)} = 1$. Then, there exists a positive constant $C$,
depending on 
$\lambda,\Lambda,n$, and the modulus of strict convexity of $u$, such that
\[ \|D^2u\|_{C^{\alpha}(B_{1/2})} \leq e^{\frac{C}{\alpha}\log\frac{1}{\alpha}}. \]
\end{thm}
We then show the sharpness of these estimates by constructing a family of solutions in two dimensions, with a fixed modulus of strict convexity,
for which $\omega$ behaves in a highly nonlinear way in both its arguments. 
We show:
\begin{thm}
\label{thm:negative}
For any $\alpha\in (0,1)$ 
there exist a family $u_r:\mathbb R^2\to \mathbb R$ of $C^{2,\alpha}$ convex functions with a uniform modulus of strict convexity and a family $f_r:\mathbb R^2 \to \mathbb R$
of $C^{\alpha}$ functions such that
$$
\det D^2u_r = f_r\qquad \text{in } \mathbb{R}^2
$$
with
$$
0< \lambda \leq f_r \leq \Lambda<\infty \qquad\text{and}\qquad \lim_{r\to 0} \|f_r\|_{C^{\alpha}(B_{1})}=\infty,
$$
and
$$
\|D^2u_r\|_{C^{\alpha}(B_{1/2})}\geq c\|f_r\|_{C^{\alpha}(B_{1})}^\rho
$$
for some $\rho>1$.

Furthermore, there exist a family $u_\alpha:\mathbb R^2\to \mathbb R$ of $C^{2,\alpha}$ convex functions with a uniform modulus of strict convexity and a family $f_\alpha:\mathbb R^2 \to \mathbb R$ of $C^{\alpha}$ functions, with $\alpha \in (0,1/2]$, such that
$$
\det D^2u_\alpha = f_\alpha\qquad \text{in } \mathbb{R}^2
$$
with
$$
0< \lambda \leq f_\alpha \leq \Lambda<\infty\qquad \text{and}\qquad \|f_\alpha\|_{C^{\alpha}(B_{1})}\leq C,
$$
and
$$
\|D^2u_\alpha\|_{C^{\alpha}(B_{1/2})}\geq e^{c/\alpha}.
$$
\end{thm}

The basic reason for the lack of linearity in \eqref{eqn:omega} can be explained as follows: using strict convexity and the affine invariance of \eqref{eqn:MA}, one can rescale to a situation in which $u$ is very close to $|\cdot|^2/2$. 
Then, the equation is linearized around this paraboloid and behaves essentially like the Laplace equation. 
However, because of the degeneracy of the Monge-Amp\`ere equation (see \cite[Section 1.1]{F} for a general discussion on this point), the geometry of a solution may become very eccentric before linearity kicks in, and even if one fixes a modulus of strict convexity, heuristic computations suggest that the eccentricity of the solution where $f$ is close to $1$ depends in a nonlinear way on $\|f\|_{C^{\alpha}(B_1)}$.

Our examples in two dimensions will show that this indeed happens\,---\,the strategy here is based on constructing solutions $u$ that are invariant under rescalings that change the axes by different powers:
\begin{equation}
\label{eqn:model}
u(x_1,x_2) = \frac{1}{\kappa^{\gamma+1}} u(\kappa x_1, \kappa^{\gamma}x_2) \qquad\text{for}\qquad \gamma > 1
\end{equation}
(see \cite{S}, \cite{W1}, and \cite{W2} for similar constructions). 
The geometry of such a $u$ near the origin is very eccentric, and one can see that $D^2u$ oscillates much more than $f : =\det D^2u$ at comparable scales. 
Yet, $f$ is invariant under the same rescaling as $u$ and is thus discontinuous at the origin. To correct for this and complete the example, 
we make a small perturbation of $f$ near the origin to make it H\"{o}lder continuous and approximate $u$ by a solution with the perturbed right-hand side. \\

This note is organized as follows. 
In Section~\ref{sec:Pos Res 1}, we prove Theorem~\ref{thm:positive1} showing, via a scaling argument, that if a solution to \eqref{eqn:MA} has a polynomial modulus of strict convexity, then $\omega$ depends polynomially on $\|f\|_{C^{\alpha}(B_{1})}$.
In Section~\ref{sec:Pos Res 2}, we prove Theorem~\ref{thm:positive2} exploiting some of the ideas Caffarelli developed in \cite{C1}.
In Section~\ref{sec:sharp}, we prove Theorem~\ref{thm:negative} splitting the argument in several steps. More concretely,
in Section~\ref{sec:Notation}, we establish notation useful for the following sections wherein we build our example family.
Then, in Section~\ref{sec:MainConstruction}, we construct a solution $u$ of the form \eqref{eqn:model},
prescribing suitable boundary data so that $f$ is strictly positive and continuous on $\partial B_1$.
(Actually, we prescribe boundary data that makes $f$ smooth away from the origin with a H\"{o}lder norm away from the origin depending only on $\gamma$.)
In Section~\ref{sec:Heuristics}, we perform some heuristic computations which show nonlinear dependence of $\|D^2u\|_{C^{\alpha}(B_{1/2})}$ on $\|f\|_{C^{\alpha}(B_{1})}$ and $1/\alpha$ at comparable scales.
In Section~\ref{sec:Approximation}, we slightly perturb $f$ near the origin and approximate $u$ by a solution with the perturbed right-hand side,
and finally in Section~\ref{sec:Nonlin Dependence}, we show that our approximation realizes the claimed nonlinear dependence.

\section{Proof of Theorem~\ref{thm:positive1}}
\label{sec:Pos Res 1}

Before proceeding with the proof of Theorem~\ref{thm:positive1}, we remark that $u$ has a polynomial modulus of strict convexity.  
Indeed, since normalized solutions (see \cite[Section 4.1]{F} for more details on normalized solutions) have a polynomial modulus of strict convexity with exponent $\sigma > 1$
depending only on $\l,\L$, and $n$ (see \cite{C2}
or \cite[Theorem 4.2.7]{F}), it follows from John's lemma \cite{J} and the affine invariance of the Monge-Amp\`{e}re equation
(see, for instance, \cite[Section 4.1.2]{F}) that 
there exists a constant 
$c_0>0$, depending only on $n,\l,\L$, and the modulus of strict convexity of $u$, such that, given any $x \in B_{1/2}$,
\begin{equation}
\label{eqn:StrictConvexity}
u(z) \geq u(x) + \langle \nabla u(x), (z-x) \rangle +c_0 |z-x|^{\sigma} \qquad \forall z \in B_{1/4}(x).
\end{equation}
From now on, we call ``universal'' any positive constant depending only on $\l,\L,n,$ and $c_0$.

Up to subtracting an affine function, we assume that $u(0) = 0$ is the minimum of $u$. Let 
\[ Z_h := \{u < h\}. \]
Then, by \eqref{eqn:StrictConvexity}, the sublevel sets $Z_h$ are compactly contained in $B_1$ for all $h \leq \hat{h}$ with $\hat{h} \ll 1$ universal. Also, there exist ellipsoids $\mathcal{E}_h$ centered at the origin (the John ellipsoids for $Z_h$) and a universal constant $\hat{C} \geq 1$ such that\footnote{Here $|E|$ denotes the Lebesgue measure of the set $E$.}

\[ \hat{C}^{-1}\mathcal{E}_h \subset Z_h \subset \hat{C}\mathcal{E}_h \qquad\text{and}\qquad |\mathcal{E}_h| = h^{n/2} \]
(see \cite{C2} or \cite[Lemmas A.3.6 and 4.1.6]{F}).
In particular, since $\hat{C}^{-1}\mathcal{E}_{\hat h}\subset Z_{\hat h}\subset B_1$, the volume estimate $|\mathcal{E}_{\hat h}| = \hat h^{n/2}$ implies that
$$
\hat{c}\hat h^{n/2}B_1 \subset \mathcal{E}_{\hat h}\subset \hat C Z_{\hat h}
$$
for some positive universal constant $\hat{c}$.
As $Z_h\supset h\hat{h}^{-1}Z_{\hat h}$ for any $0<h \leq \hat h$ (by the convexity of $u$) and by \eqref{eqn:StrictConvexity}, we deduce that
\begin{equation}
\label{eqn:Zh}
\hat{C}^{-1}\hat{c}\hat{h}^{\frac{n}{2}-1}hB_{1} \subset Z_h\subset c_0^{-1}h^{1/\sigma}B_1. 
\end{equation}
Moreover,
\begin{equation*}
\label{eqn:Eh}
\hat{C}^{-2}\hat{c}\hat h^{\frac{n}{2}-1}hB_1 \subset \mathcal{E}_{h}\subset \hat C c_0^{-1}h^{1/\sigma}B_{1},
\end{equation*}
and letting $L_h$ be affine maps such that $L_h(B_1) = \mathcal{E}_h$, we find that 
\begin{equation}
\label{eqn:LhInequalities}
\|L_h\| \leq \bar{c}h^{1/\sigma} \qquad\text{and}\qquad \|L_h^{-1}\| \leq \bar{C}h^{-1}
\end{equation}
with $\bar{c} = \hat{C}c_0^{-1}$ and $\bar{C} = \hat{C}^{2}\hat{c}^{-1}\hat{h}^{1-\frac{n}{2}}$.

Define 
\[ S_h := L_h^{-1}(Z_h)\qquad\text{and}\qquad u_h(z) := \frac{1}{h}u(L_h z). \] 
Clearly, $u_h$ solves
\[\begin{cases}
\det D^2u_h = f_h &\text{in } S_h \\
u_h = 1 &\text{on }\partial S_h
\end{cases}\]
where $f_h(z) := f(L_hz)$.
Observe that, by the first inequality in \eqref{eqn:LhInequalities}, we have that
$$[f_h]_{C^{\alpha}(S_h)} \leq \|L_h\|^{\alpha} [f]_{C^{\alpha}(Z_h)} \leq \bar{c}^\alpha h^{\alpha/\sigma}[f]_{C^{\alpha}(B_1)}.$$
Similarly, the second inequality in \eqref{eqn:LhInequalities} implies that
\begin{equation}
\label{eqn:bar u}
\|D^2u\|_{C^{\alpha}(Z_h/2)} \leq h\|L_h^{-1}\|^{2 + \alpha}\|D^2u_h\|_{C^{\alpha}(S_h/2)} \leq \bar{C}^{2+\alpha}h^{-(1+\alpha)}\|D^2u_h\|_{C^{\alpha}(S_h/2)}.
\end{equation}
Hence, setting 
\begin{equation}
\label{eqn:bar h}
\bar{h} := \|f\|_{C^{\alpha}(B_1)}^{-\sigma/\alpha},
\end{equation}
we deduce that
$$[f_{\bar{h}}]_{C^{\alpha}(S_{\bar{h}})} \leq \bar{c}^\alpha.$$
(Without loss of generality, we assume that $\bar{h} \leq \hat{h}$ as $\hat{h}$ is universal.)
Since the sets $S_h$ are equivalent to $B_1$ up to dilations by universal constants, 
it follows by \cite{C3} that $\|D^2u_{\bar{h}}\|_{C^{\alpha}(S_{\bar{h}}/2)}$ is bounded by a constant depending on $\l,\L,n,\alpha$, and $c_0$.
Hence, recalling \eqref{eqn:bar u} and \eqref{eqn:bar h}, we conclude that
\begin{equation}
\label{eqn:sectional est}
\|D^2u\|_{C^{\alpha}(Z_{\bar{h}}/2)} \leq \tilde{C}\|f\|_{C^{\alpha}(B_1)}^{\frac{\sigma(1+\alpha)}{\alpha}}
\end{equation}
for a positive constant $\tilde{C}$ depending on $\l,\L,n,\alpha$, and $c_0$.
Notice that the same argument can be used at any other point $x \in B_{1/2}$
to show that the above estimate holds with the section 
$$
Z_{\bar{h}}(x):=\{z \,:\, u(z)<u(x)+\langle \nabla u(x),z-x\rangle +\bar h\} 
$$
in place of $Z_{\bar h}=Z_{\bar h}(0)$.

To conclude the proof we notice that, given any two points $x,y \in B_{1/2}$,
for every $N>1$, we can find a sequence of points $\{x_i\}_{i=0}^N$ on the segment joining $x$ and $y$ such that 
$$
x_0=x,\qquad x_N=y,\qquad\text{and}\qquad |x_{i+1}-x_i|=\frac{|x-y|}{N}< \frac{1}{N}.
$$
In particular, if we choose $N=2\hat{C}/(\hat{c}\hat{h}^{\frac{n}{2}-1}\bar{h})$, it follows by \eqref{eqn:Zh}
that the sections $Z_{\bar{h}}(x_i)/2$ cover the segment $[x,y]$.
And so, applying \eqref{eqn:sectional est}
to each of these sections and recalling our choice of $\bar h$ (see \eqref{eqn:bar h}), we obtain that
\begin{align*}
|D^2u(x)-D^2u(y)| &\leq \sum_{i=0}^{N-1}|D^2u(x_{i+1}) - D^2u(x_i)|
\leq \tilde C\|f\|_{C^{\alpha}(B_1)}^{\frac{\sigma(1+\alpha)}{\alpha}}
\,\sum_{i=0}^{N-1}|x_{i+1} - x_i|^\alpha \\
&=\tilde C\|f\|_{C^{\alpha}(B_1)}^{\frac{\sigma(1+\alpha)}{\alpha}}N^{1-\alpha}|x-y|^\alpha 
\leq \frac{\tilde C(2\hat{C})^{1-\alpha}}{(\hat c\hat{h}^{\frac{n}{2}-1})^{1-\alpha}}\|f\|_{C^{\alpha}(B_1)}^{{2\sigma}/{\alpha}}|x-y|^\alpha,
\end{align*}
which proves that
\[ \|D^2u\|_{C^{\alpha}(B_{1/2})} \leq C\|f\|_{C^{\alpha}(B_1)}^{{2\sigma}/{\alpha}} \]
for some positive constant $C$ depending on $\lambda,\Lambda,n,\alpha$, and the modulus of strict convexity of $u$,
as desired.

\section{Proof of Theorem~\ref{thm:positive2}}
\label{sec:Pos Res 2}

We begin with a pair of lemmas in which we assume that $u$ is a strictly convex solution to
\[ \det D^2u = f \qquad \text{in } B_1 \subset \mathbb{R}^n \]
with $f \in C^{\alpha}(B_1)$.
As we are interested in the behavior of the $C^{2,\alpha}$ norm of solutions as $\alpha \to 0$, we assume that $\alpha \in (0,1/2]$.

In the following, we let $C_n$ denote a positive dimensional constant that may change from line to line.

\begin{lem}
\label{lem:basic}
Let $\hat{\delta}  \leq 10^{-2}$. Assume that $u(0) = 0$ is the minimum of $u$. There exists a positive dimensional constant $\hat{\epsilon}$ such that 
the following holds:
if $\hat{P}$ is a convex quadratic polynomial with $\det D^2\hat{P} = 1$ such that
\begin{equation}
\label{eqn:basic hyp0}
|D^2\hat{P} - \Id| \leq 1/2,
\end{equation}
\begin{equation}
\label{eqn:basic hyp1} 
\|u-\hat{P}\|_{L^{\infty}(B_1)} \leq \hat{\delta},
\end{equation}
and
\begin{equation}
\label{eqn:basic hyp2}
\|f-1\|_{L^{\infty}(B_1)} \leq \hat{\delta}\hat{\epsilon},
\end{equation}
then there exists a convex quadratic polynomial $\hat{Q}$ with $\det D^2\hat{Q} = 1$ such that
\[ |D^2\hat{Q} - D^2\hat{P}| \leq \hat{C}\hat{\delta} \]
and
\[ \|u-\hat{Q}\|_{L^{\infty}(B_{\hat r})} \leq \hat{\delta}\hat{r}^{2+\alpha} \]
for some positive dimensional constants $\hat{C}$ and $\hat{r}$.
\end{lem}

\begin{proof}
Since by assumption $\hat\delta\leq 10^{-2}$,
by \eqref{eqn:basic hyp0} and \eqref{eqn:basic hyp1}, there exist positive dimensional constants $\tilde{r}$ and $\tilde{h}$ such that $B_{\tilde{r}} \subset Z_{\tilde{h}} := \{ u < \tilde{h} \} \subset B_1$.
Let $w$ solve
\[\begin{cases}
\det D^2w = 1 &\text{in } Z_{\tilde{h}} \\ 
w = \tilde{h} &\text{on } \partial Z_{\tilde{h}}.
\end{cases}\]
Then, considering the sub and supersolutions $(1\pm\hat{\delta}\hat{\epsilon})^{1/n}[w-\tilde{h}]$, the comparison principle
(see, for instance, \cite[Theorem 2.3.2]{F}) implies that 
\[ (1+\hat{\delta}\hat{\epsilon})[w-\tilde{h}] \leq u-\tilde{h} \leq (1-\hat{\delta}\hat{\epsilon})[w-\tilde{h}] \qquad \text{in } Z_{\tilde{h}}. \]
Therefore,
\begin{equation}
\label{eqn:max principle}
\|u-w\|_{L^{\infty}(Z_{\tilde{h}})} \leq C_n\hat{\delta}\hat{\epsilon}.
\end{equation}
Also, it follows by \eqref{eqn:basic hyp0} and Pogorelov's interior regularity estimates (see \cite{P} or \cite[Theorems 3.3.1
and 3.3.2]{F}) that $w$ is uniformly convex and $C^\infty$ inside $B_{\tilde r}\subset Z_{\tilde h}$ with bounds depending only on dimension inside 
$B_{\tilde r/2}$.

We now observe that, since $\det D^2\hat P=1$ and
defining $ v_t := tw+(1-t)\hat{P}$,
$$
0=\det D^2 w - \det D^2 \hat P=\int_0^1\frac{d}{dt}\det D^2 v_t\,dt
= \trace \big(A \cdot D^2(w-\hat P) \big)
$$
with
\[ A := \int_{0}^{1} \det D^2v_t (D^2v_t)^{-1} \, dt. \]
That is,
$w-\hat{P}$ solves the linear equation $\trace\big(A \cdot D^2(w-\hat{P}) \big) = 0$ in $Z_{\tilde{h}}$.
Since the matrix $A$ is smooth and uniformly elliptic 
inside $B_{\tilde{r}/2}$, it follows by elliptic regularity that
\[ \|w - \hat{P}\|_{C^{3}(B_{\tilde{r}/4})} \leq C_n\|w - \hat{P}\|_{L^{\infty}(B_{\tilde{r}/2})} \] 
(see \cite[Chapter 6]{GT}).
From \eqref{eqn:basic hyp1} and \eqref{eqn:max principle}, we notice that
\[ \|w-\hat{P}\|_{L^{\infty}(B_{\tilde{r}/2})} \leq \|u-\hat{P}\|_{L^{\infty}(B_{\tilde{r}/2})} + \|u-w\|_{L^{\infty}(B_{\tilde{r}/2})} \leq C_n\hat{\delta}, \]
and so 
\[ \|w - \hat{P}\|_{C^{3}(B_{\tilde{r}/4})} \leq C_n\hat{\delta}. \]
Let $\hat{Q}$ be the second order Taylor polynomial of $w$ centered at the origin.
Then, for all $r \leq \tilde{r}/4$,
\begin{equation}
\label{eqn:Taylor}
\|w-\hat{Q}\|_{L^{\infty}(B_r)} \leq \|D^3w\|_{L^{\infty}(B_r)}r^3= \|D^3w-D^3\hat{P}\|_{L^{\infty}(B_r)}r^3 \leq C_n\hat{\delta}r^3.
\end{equation}
Also,
\[ |D^2\hat{Q} - D^2\hat{P}| = |D^2w(0) - D^2\hat{P}| \leq \|D^2w-D^2\hat{P}\|_{L^{\infty}(B_{\tilde{r}/4})} \leq C_n\hat{\delta}. \]
Therefore, by \eqref{eqn:max principle} and \eqref{eqn:Taylor}, for all $r \leq \tilde{r}/4$,
\[ \|u-\hat{Q}\|_{L^{\infty}(B_r)} \leq \|u-w\|_{L^{\infty}(B_r)} + \|w-\hat{Q}\|_{L^{\infty}(B_r)} \leq 
C_n\hat{\delta}\big(\hat{\epsilon} + r^{3}\big). \]
Recalling that by assumption $\alpha \leq 1/2$, taking $\hat{r} = \min\{\tilde{r}/4, 1/(2{C}_n)^{2}\}$ and then $\hat{\epsilon} = \hat{r}^{3}$ proves the result.
\end{proof}

\begin{lem}
\label{lem:iteration}
Assume that $f(0) =1$ and $u(0) = 0$ is the minimum of $u$. There exist positive dimensional constants $\bar{C}, \bar{\epsilon}$, and $\bar{\delta}$ such that if 
\[ \|u - |x|^2/2\|_{L^{\infty}(B_1)} \leq \bar{\delta}\alpha, \]
and
\[ \|f-1\|_{C^{\alpha}(B_1)} \leq \bar{\delta}\alpha\bar{\epsilon}, \]
then there exists a convex quadratic polynomial $Q$ with $\det D^2Q = 1$ such that 
\[ |D^2Q| \leq \bar{C} \]
and
\[ \|u - Q\|_{L^{\infty}(B_r)} \leq \bar{C}\bar{\delta} r^{2 + \alpha} \qquad \forall r < 1. \]
\end{lem}

\begin{proof}
Let $\bar \delta \leq 10^{-2}$ to be chosen. Taking $\hat{P} = |x|^{2}/2$, we apply Lemma~\ref{lem:basic} with $\hat{\delta} = \bar \delta\alpha$ to obtain positive dimensional constants $\hat{C},\hat{r}$, and $\hat{\epsilon}$ and a convex quadratic polynomial $\hat{Q}_1$ with $\det D^2\hat{Q}_1 = 1$ such that 
\[ |D^2\hat{Q}_1 - \Id| \leq \hat{C}\bar \delta\alpha \]
and
\begin{equation}
\label{eq:u Q1}
\|u-\hat{Q}_1\|_{L^{\infty}(B_{\hat r})} \leq \bar \delta\alpha\hat{r}^{2+\alpha} 
\end{equation}
provided $\|f-1\|_{L^{\infty}(B_1)} \leq \|f-1\|_{C^{\alpha}(B_1)} \leq \bar \delta\alpha\hat{\epsilon}$.

Let $f_2(x) := f(\hat{r}x)$, $u_2$ and $\hat{P}_2$ be the quadratic rescalings of $u$ and $\hat{Q}_1$ by $\hat{r}$ respectively, that is 
\[ u_2(x) := \frac{1}{\hat{r}^2}u(\hat{r}x) \qquad\text{and}\qquad \hat{P}_2(x) := \frac{1}{\hat{r}^2}\hat{Q}_1(\hat{r}x), \]
and choose $\bar \delta=\min\{10^{-2},1/2\hat{C}\}$.
Observe that $u_2,\hat{P}_2$, and $f_2$ satisfy the hypotheses of Lemma~\ref{lem:basic} with $\hat{\delta} =\bar \delta\alpha\hat{r}^{\alpha}$. 
Indeed, \eqref{eqn:basic hyp0} is satisfied as we have chosen $\hat{C}\bar\delta\alpha \leq 1/2$, $\det D^2\hat{P}_2 =1$ holds by construction, 
\eqref{eqn:basic hyp1} follows from \eqref{eq:u Q1}, and \eqref{eqn:basic hyp2} follows since we have assumed that $\|f-1\|_{C^{\alpha}(B_1)} \leq \bar{\delta}\alpha\hat{\epsilon}$ and as $f(0)=1$.
So, applying Lemma~\ref{lem:basic}, we find a convex quadratic polynomial $\hat{Q}_2$ with $\det D^2\hat{Q}_2 = 1$ such that
\[ |D^2\hat{Q}_2 - D^2\hat{P}_2| \leq \hat{C}\bar\delta\alpha\hat{r}^{\alpha} \]
and
\[ \|u_2-\hat{Q}_2\|_{L^{\infty}(B_{\hat r})} \leq \bar\delta\alpha\hat{r}^{\alpha}\hat{r}^{2+\alpha}. \]
Also, we see that
\[ |D^2\hat{Q}_2 - \Id| \leq |D^2\hat{Q}_2 - D^2\hat{P}_2| + |D^2\hat{P}_2 - \Id| \leq \hat{C}\bar\delta\alpha(1 + \hat{r}^{\alpha}) \]
and
\[ \|u-\hat{r}^2\hat{Q}_2(\cdot/\hat{r})\|_{L^{\infty}(B_{\hat r^2})} \leq \bar\delta\alpha\hat{r}^{2(2 + \alpha)}. \]
Hence, for any $k \geq 1$, we let $f_k(x) := f_{k-1}(\hat{r}x)$ and $u_k$ and $\hat{P}_k$ be the quadratic rescalings of $u_{k-1}$ and $\hat{Q}_{k-1}$ by $\hat{r}$ respectively. Then, provided that
 $\bar{\epsilon} \leq  \hat{\epsilon}$ and
\[ \hat{C}\bar\delta \alpha(1 + \hat{r}^{\alpha} + \cdots + \hat{r}^{k\alpha}) \leq \frac{\hat{C}\bar\delta\alpha}{1-\hat{r}^{\alpha}} \leq \frac{1}{2}, \]
we can iteratively obtain convex quadratic polynomials $\hat{Q}_k$ with $\det D^2\hat{Q}_k = 1$ such that
\[ |D^2\hat{Q}_k - D^2\hat{P}_k| \leq \hat{C}\bar\delta\alpha\hat{r}^{(k-1)\alpha}, \]
\[ |D^2\hat{Q}_k - \Id| \leq \frac{1}{2}, \] 
and
\begin{equation}
\label{eq:u Qk}
\|u - \hat{r}^{2(k-1)}\hat{Q}_k(\cdot/\hat{r}^{k-1})\|_{L^{\infty}(B_{\hat r^k})} \leq \bar\delta\alpha\hat{r}^{k(2+\alpha)}.
\end{equation}
In particular, we can simply choose $\bar\epsilon=\hat \epsilon$.

We claim that the sequence 
\[ Q_k(x) := \hat{r}^{2(k-1)}\hat{Q}_k(x/\hat{r}^{k-1}) \]
converges to the desired $Q$.
Indeed, since each $Q_k$ is a quadratic polynomial, there exist triplets $(c_k,b_k,A_k) \in \mathbb{R}\times\mathbb{R}^{n}\times\mathbb{R}^{n\times n}$ such that $Q_k(x) = c_k + \langle b_k, x \rangle + \langle A_kx , x \rangle$. 
By \eqref{eq:u Qk}, we see that
$$
\|Q_{k}-Q_{k+1}\|_{L^{\infty}(B_{\hat r^{k+1}})}\leq \|u - Q_k\|_{L^{\infty}(B_{\hat r^k})}+\|u - Q_{k+1}\|_{L^{\infty}(B_{\hat r^{k+1}})}
\leq 2\bar\delta\alpha\hat{r}^{k(2+\alpha)},
$$
which implies that
\[ |c_k - c_{k+1}|, \, \hat{r}^k|b_k - b_{k+1}|, \, \hat{r}^{2k}|A_k - A_{k+1}| \leq C_n \bar{\delta}\alpha\hat{r}^{k(2+\alpha)}. \]
Hence,
\[ \|Q_k - Q_{k+1}\|_{L^{\infty}(B_{r})} \leq C_n\bar{\delta}\alpha\hat{r}^{k\alpha}\max\big\{\hat{r}^{2k},r^{2}\big\} 
\qquad \forall r \leq 1. \]
Summing these estimates we obtain that, for any $j \geq 1$,
\[
\|u-Q\|_{L^{\infty}(B_{\hat{r}^j})} \leq \|u - Q_j\|_{L^{\infty}(B_{\hat{r}^j})} + \sum_{k \geq j} \|Q_{k}-Q_{k+1}\|_{L^{\infty}(B_{\hat{r}^j})} 
\leq C_n\bar{\delta}\hat{r}^{j(2 + \alpha)},
\]
which proves the desired result.
\end{proof}

In the following proof, we call ``universal'' any positive constant depending only on 
$\l,\L,n$, and the modulus of strict convexity of $u$.
In particular, we let $C$ denote a positive universal constant that may change from line to line.

\begin{proof}[Proof of Theorem~\ref{thm:positive2}]
Without loss of generality, we assume that $f(0) = 1$ and $u(0) = 0$ is the minimum of $u$. 
Let $Z_h, L_h, S_h, u_h$, and $f_h$ be defined as in the proof of Theorem~\ref{thm:positive1}, and, via a similar argument to that in Theorem~\ref{thm:positive1}, observe that for $h > 0$ such that $Z_h \subset \subset B_1$, we have
\[ \|f_h - 1\|_{C^{\alpha}(S_h)} \leq Ch^{\alpha/\sigma}. \]
Here $\sigma > 1$ depends only on $n,\l$, and $\L$, and $C \geq 1$ is universal (see \cite{C2}).

Let $w_h$ be the solutions to 
\[\begin{cases}
\det D^2w_h = 1 &\text{in } S_h \\
w_h = 1 &\text{on } \partial S_h 
\end{cases}\]
and $Q_h$ be the second order Taylor polynomials of $w_h$ centered at the origin.
By the same techniques used in Lemma~\ref{lem:basic}, we deduce that
\[ \|u_h - Q_h\|_{L^{\infty}(B_r)} < C(h^{\alpha/\sigma} + r^3). \]
If we let $u_{r,h}$ and $Q_{r,h}$ be the quadratic rescalings of $u_h$ and $Q_h$ by $r$, then
\[ \|u_{r,h}-Q_{r,h}\|_{L^{\infty}(B_1)} < C\bigg(\frac{h^{\alpha/\sigma}}{r^2} +r\bigg). \]

After an affine transformation taking $Q_{r,h}$ to $|\cdot|^2/2$, we are in the setting of Lemma~\ref{lem:iteration} provided that $r = \bar{\delta}\alpha/2C$ and $h^{\alpha / \sigma} = (\bar{\delta}\alpha/2C)^3$. 
Consequently,
$$\|D^2u_{r,h}\|_{C^{\alpha}(B_{1/2})} \leq C.$$
Scaling back we obtain that
\begin{align*}
\|D^2u\|_{C^{\alpha}(Z_h/2)} \leq Cr^{-\alpha}h\|L_h^{-1}\|^{2 + \alpha} \leq Cr^{-\alpha} h^{-(1+\alpha)} \leq \bigg(\frac{1}{\alpha}\bigg)^{\frac{C\sigma}{\alpha}}.
\end{align*}
Then, a covering argument like the one in the proof of Theorem \ref{thm:positive1} yields
$$
\|D^2u\|_{C^{\alpha}(B_{1/2})} \leq \bigg(\frac{1}{\alpha}\bigg)^{\frac{C\sigma}{\alpha}},
$$
which implies the claimed estimate.
\end{proof}

\section{Proof of Theorem \ref{thm:negative}}
\label{sec:sharp}

\subsection{Notation}
\label{sec:Notation}

It is natural to work in a geometry suited to the scaling invariance of the model solution \eqref{eqn:model}. Hence, we define the linear transformation $A_r$ by 
$$A_r(x_1,x_2) := \left(r^{\frac{1}{\gamma+1}}x_1, r^{\frac{\gamma}{\gamma+1}}x_2\right).$$
Furthermore, let $I := [-1,1]$, $Q_1 := I \times I$, and define the rectangles $Q_r$ by
$$Q_r := A_r(Q_1).$$
Note that the area of $Q_r$ is $4r$, and that $Q_r$ is much longer horizontally than vertically for $r$ small.

For a convex function $u$ with homogeneity \eqref{eqn:model}, we have that 
\begin{equation}
\label{eqn:scaling}
u(x) = \frac{1}{r}u(A_{r}x). 
\end{equation}
Thus, we see that $A_{r}(\{u < \tau\}) = \{u < r\tau\}$, and so the sublevel set of $u$ of height $r$ is equivalent to $Q_{r}$ up to dilations. More precisely,
\[ cQ_{r} \subseteq \{u<r\} \subseteq CQ_{r} \quad \forall r >0
\qquad\text{whenever}\qquad cQ_{1} \subseteq \{u<1\} \subseteq CQ_{1}, \] 
for positive constants $c$ and $C$. 

We use the convention $s^{\beta} = |s|^{\beta}$ for all $s,\beta \in \mathbb{R}$.

In the following sections, any constant called ``universal'' is one depending only on $\gamma$.
From this point forward we let $c$ and $C$ be positive universal constants, and note that they may change from line to line.

\subsection{Construction of $u$}
\label{sec:MainConstruction}

We construct a solution $u$ defined with the homogeneity \eqref{eqn:model}:
\[ u(x_1,x_2) = \frac{1}{\kappa^{\gamma+1}}u(\kappa x_1,\kappa^{\gamma}x_2) \qquad\text{for}\qquad \gamma > 1. \]
It is immediate to check that $\det D^{2}u$ is constant along curves $x_{2} = mx_{1}^{\gamma}$ for $m \in \mathbb{R}$, and it is uniquely determined once understood along horizontal or vertical lines. Thus, we have two choices:
\[ u(x_1,x_2) = x_2^{\frac{\gamma + 1}{\gamma}}g\big(x_2^{-1/\gamma}x_1\big) \qquad\text{or}\qquad u(x_1,x_2) = x_1^{\gamma+1}\tilde{g}\big(x_1^{-\gamma}x_2\big) \]
where $g,\tilde{g} : \overline{\mathbb{R}} \rightarrow \mathbb{R}^{+}$ are convex, even functions. This can be seen using the homogeneity enjoyed by $u$ and setting
\[ u(t,\pm 1) = g(t) \qquad\text{and}\qquad u(\pm 1,t) = \tilde{g}(t) \]
with $t = x_2^{-1/{\gamma}}x_1$ in the first case and $t = x_1^{-\gamma}x_2$ in the second. Consequently, we see that either
\[ \det D^2u(x_{1},x_{2}) = F[g,\gamma](t) \qquad\text{if } t = x_2^{-1/{\gamma}}x_1 \]
or
\[ \det D^{2}u(x_{1},x_{2}) = \tilde{F}[\tilde{g},\gamma](t) \qquad\text{if } t = x_1^{-\gamma}x_2\]
where
\[ F[g,\gamma](t) := \gamma^{-2}g''(t)\Big((\gamma + 1)g(t) + (\gamma-1)tg'(t)\Big) - g'(t)^2 \]
and
\[ \tilde{F}[\tilde{g},\gamma](t) := \gamma \tilde{g}''(t)\Big((\gamma + 1)\tilde{g}(t)-(\gamma - 1)t\tilde{g}'(t)\Big) - \tilde{g}'(t)^2. \]
Notice that once either $g$ or $\tilde{g}$ is prescribed the other is determined. Indeed,
\begin{equation}
\label{eqn:matching} 
\tilde{g}(t) = t^{\frac{\gamma + 1}{\gamma}}g\big(t^{-1/\gamma}\big).
\end{equation}
With this in mind, we restrict our attention to $g: \overline{\mathbb{R}} \rightarrow \mathbb{R}^{+}$ and let
\[ u(x_{1},x_{2}) = x_2^{\frac{\gamma + 1}{\gamma}}g\big(x_2^{-1/\gamma}x_1\big). \]

We build $g$ so that $\det D^2u$ is strictly positive, smooth away from the origin, and has a universal Lipschitz constant outside $Q_{1}$. Define
\[ g_0(t) := \frac{1}{\gamma + 1}\big(1 + |t|^{\gamma + 1}\big). \]
A simple computation shows that
\[ F[g_0,\gamma](t) = \gamma^{-1}|t|^{\gamma - 1}, \]
which is nonnegative. 

First, we modify $g_{0}$ near zero to make $F$ strictly positive. We do this by replacing $g_{0}$ by a parabola on a small interval $[-t_{0},t_{0}]$. Thus, define 
\[ g_{1}(t) :=
\begin{cases}
g_{0}(t) &\text{in } \overline{\mathbb{R}} \setminus [-t_0,t_0] \\
\frac{1}{\gamma + 1}(a + bt^2) &\text{in }[-t_0,t_0]
\end{cases} \]
with
\begin{equation*}
a =  1 - \frac{\gamma - 1}{2}t_{0}^{\gamma + 1} \qquad\text{and}\qquad b = \frac{\gamma + 1}{2} t_{0}^{\gamma-1}.
\end{equation*}
Choosing $a$ and $b$ in this way, we have that $g_{1}$ is locally $C^{1,1}$. Moreover, in $[-t_0,t_0]$,
\begin{equation*}
\label{eqn:FirstParabola}
\begin{split}
F[g_1,\gamma](t) &= \frac{2(\gamma + 1)ab - 2(2\gamma-1)(\gamma - 1)b^{2} t^{2}}{\gamma^{2}(\gamma + 1)^{2}} \\
&\geq ct_{0}^{\gamma - 1}\Big[1 - Ct_{0}^{\gamma + 1}\Big] .
\end{split}
\end{equation*}
In particular, choosing $t_{0}$ sufficiently small makes $F[g_{1},\gamma]$ strictly positive. 

Second, we modify $g_{1}$ at infinity so that $F$ is bounded. By \eqref{eqn:matching}, modifying $g_{1}$ at infinity corresponds to modifying 
\[ \tilde{g}_{1}(t) = \tilde{g}_{0}(t) = \frac{1}{\gamma + 1}\Big(1 + |t|^{\frac{\gamma + 1}{\gamma}}\Big) \qquad\forall |t| < t_{0}^{-\gamma} \]
at the origin. To this end, we replace $\tilde{g}_{1}$ by a parabola on a small interval $[-\tilde{t}_0,\tilde{t}_0]$ and match the functions and derivatives at $\pm \tilde{t}_{0}$. More specifically, define
\[ \tilde{g}_{1}(t) :=
\frac{1}{\gamma + 1}\big(\tilde{a} + \tilde{b} t^2\big) \qquad \forall t \in [-\tilde{t}_0,\tilde{t}_0] \]
with
\[ \tilde{a} = 1 + \frac{\gamma - 1}{2\gamma} \tilde{t}_0^{\frac{\gamma + 1}{\gamma}} \qquad\text{and}\qquad \tilde{b} = \frac{\gamma + 1}{2\gamma} \tilde{t}_{0}^{-\frac{\gamma-1}{\gamma}}. \]
It follows that in $[-\tilde{t}_0,\tilde{t}_0]$,
\begin{equation*}
\label{eqn:SecondParabola}
\begin{split}
\tilde{F}[\tilde{g_1},\gamma](t) 
&\geq c\tilde{t}_{0}^{-\frac{\gamma-1}{\gamma}}\Big[1 - C\tilde{t}_{0}^{\frac{\gamma + 1}{\gamma}}\Big] ,
\end{split}
\end{equation*}
and if we take $\tilde{t}_0$ small enough, then $\tilde{F}[\tilde{g}_{1},\gamma]$ is strictly positive. Therefore, if we define
\[ g_{2}(t) := 
\begin{cases}
g_{1}(t) &\text{in } \big[-\tilde{t}_{0}^{-1/\gamma},\tilde{t}_{0}^{-1/\gamma}\big] \\
\frac{1}{\gamma + 1}\big(\tilde{a}|t|^{\gamma + 1} + \tilde{b}|t|^{-(\gamma - 1)}\big) &\text{in } \overline{\mathbb{R}} \setminus \big[-\tilde{t}_{0}^{-1/\gamma},\tilde{t}_{0}^{-1/\gamma}\big],
\end{cases} \]
then we have that $F[g_{2},\gamma]$ is strictly positive and bounded. Moreover, $g_{2}$ is $C^{1,1}$. 

Finally, we modify $g_{2}$ to make $F[g_{2},\gamma]$ smooth. At the moment, the second derivative of $g_{2}$ jumps on the set
\[ D := \Big{\{}\pm t_0, \pm \tilde{t}_0^{-1/\gamma}\Big{\}}. \] 
Let $\eta$ be a smooth cutoff that is identically one around the points in $D$ and zero otherwise, let $g_{2}^{\epsilon}$ be a standard mollification of $g_{2}$, and define
\[ g := g_{2} + \eta(g_{2}^{\epsilon} - g_{2}). \] 
When we do this, 
\[ |g(t) - g_{2}(t)|,\,|g'(t) - g_{2}'(t)| \leq C\epsilon \qquad \forall t \in \overline{\mathbb{R}} \]
since $g_{2}$ and $g_{2}'$ are Lipschitz and only differ from $g$ and $g'$ respectively in a neighborhood of $D$. Moreover, away from $D$ the same is true for the second derivatives of $g_{2}$ and $g$. Furthermore, the smaller value of $g_{2}''$ immediately to the left or right of a discontinuity improves under smoothing.  So, near a discontinuity the values of $F[g,\gamma]$ are at least the smaller value of $F[g_{2},\gamma]$ up to small error from $\|g - g_{2}\|_{L^{\infty}(\overline{\mathbb{R}})}$ and $\|g' - g_{2}'\|_{L^{\infty}(\overline{\mathbb{R}})}$.

With this choice of $g$ we obtain a solution $u$ with the desired homogeneity and such that $f := \det D^2u$ is smooth away from the origin, strictly positive, bounded, Lipschitz on $\partial Q_1$, and invariant under $A_{r}$:
\[ f(A_{r}x) = f(x). \]

\subsection{Heuristic Computations}
\label{sec:Heuristics}

Noticing that $\|f\|_{C^{\alpha}(\mathbb{R}^{2} \setminus Q_{1})} \leq C$ and $\partial_{22}u\sim 1$ in $Q_2 \setminus Q_1$, it follows by scaling that 
$$\|f\|_{C^{\alpha}(Q_1 \setminus Q_r)} \lesssim r^{-\frac{\alpha \gamma}{\gamma+1}}\qquad \forall r\in (0,1)$$
and
$$\|D^2u\|_{C^{\alpha}(Q_\tau \setminus Q_{\tau/2})} \sim \tau^{-\frac{\gamma - 1 + \alpha \gamma}{\gamma + 1}}\qquad \forall\tau\in (0,1)$$
(see the next section for details).

Now, given $r>0$ small, we will see that we can well-approximate $u$ by a solution $u_r$ with right-hand side $f_r$ that coincides with $f$ outside $Q_{r}$ and
is a smoothed version of $f$ inside $Q_r$. We will be able to prove that $u$ and $u_r$ are close on the scale $\tau \sim {r^{1/2}}$. In particular, $u_r$ will still satisfy
$$\|D^2u_r\|_{C^{\alpha}(B_{1/2})}\geq \|D^2u_r\|_{C^{\alpha}(Q_{\sqrt{r}} \setminus Q_{\sqrt{r}/2})} \gtrsim r^{-\frac{\gamma - 1 + \alpha\gamma}{2(\gamma + 1)}}.$$
Now, by fixing $\gamma$ such that
\[ \frac{\gamma - 1}{\gamma} > \alpha \]
and then taking $r \rightarrow 0$, we expect $\|D^2 u_r\|_{C^{\alpha}(B_{1/2})}$ to grow faster than a polynomial of degree 
\[ \frac{1}{2}\left(1 + \frac{\gamma - 1}{\alpha\gamma}\right) > 1 \]
in $\|f_r\|_{C^{\alpha}(B_1)}$.
Similarly, if we fix $\gamma$ of order $1$ and take 
\[ \alpha \sim \frac{1}{|\log r|} \qquad\text{so that}\qquad \|f\|_{C^{\alpha}(B_{1})} \sim 1, \] 
computations suggest that $\|D^2 u_r\|_{C^{\alpha}(B_{1/2})}$ grows exponentially in $1/\alpha$.
We now provide all the details.

\subsection{Construction of $u_r$ and $f_r$}
\label{sec:Approximation}
We first construct $f_r$ as follows: we set $f_{r} = f$ outside of $Q_r$,
while inside $Q_r$ we take $f_{r}$ to be an appropriate rescaling of $\tilde{F}[\tilde{g},\gamma]$ on vertical lines in $Q_{r}$
defined by
\[ f_r(x_1,x_2) := \frac{1}{\tilde{F}(1)}F\Big(r^{-\frac{1}{\gamma+1}} x_1\Big)\tilde{F}\Big(r^{-\frac{\gamma}{\gamma+1}}x_2\Big) \qquad \forall (x_1,x_2) \in Q_r, \]
where $F = F[g,\gamma]$ and $\tilde{F} = \tilde{F}[\tilde{g},\gamma]$. (Recall that $F(1) = \tilde{F}(1)$.) By scaling, we have that
\begin{equation}
\label{eqn:RHSCalpha}
\|f_r\|_{C^{\alpha}(Q_1)} \leq Cr^{-\frac{\alpha \gamma}{\gamma + 1}}.
\end{equation}
Indeed, notice that $f_{r}(A_{\kappa}x) = f_{r/\kappa}(x)$ and $\|f_1\|_{C^{\alpha}(\mathbb{R}^{2})} \leq C$ by the invariance of $f$ under $A_{r}$. Thus, for $p,q \in Q_1$, we have that
$$|f_r(p) - f_r(q)| = |f_1(A_{1/r}p) - f_1(A_{1/r}q)| \leq Cr^{-\frac{\gamma\alpha}{\gamma+1}}|p-q|^{\alpha},$$
verifying inequality \eqref{eqn:RHSCalpha}.

Next, let $u_r$ be the solution to
\[\begin{cases}
\det D^2u_r = f_r &\text{in } B_{1} \\ 
u_r  = u &\text{on } \partial B_{1}.
\end{cases}\]
Then, $\varphi := u_r - u$ solves the linear equation
\[\begin{cases}
\trace \big(A \cdot D^2\varphi \big) = f_r - f &\text{in } B_1 \\
\varphi = 0 &\text{on } \partial B_1
\end{cases}\]
where
\[ A := \int_{0}^{1} \det D^2v_t (D^2v_t)^{-1} \, dt \qquad\text{and}\qquad v_t := tu_r+(1-t)u. \]
By the concavity of $(\det)^{1/2}$ on symmetric, positive semi-definite $2 \times 2$ matrices, we have
\begin{align*}
(\det A)^{1/2} \geq \int_{0}^1 \left(t\left(\det D^2u_r\right)^{1/2} + (1-t)\left(\det D^2u\right)^{1/2}\right)\,dt \geq c.
\end{align*}
Since $|f_r - f| \leq C\1_{Q_r},$ the ABP estimate (see \cite[Chapter 9]{GT}) implies that
\begin{equation}
\label{eqn:ABP}
\|u-u_r\|_{L^{\infty}(B_1)} \leq C \|f-f_r\|_{L^2(B_1)} \leq  Cr^{1/2}.
\end{equation}

\begin{rmk}
Once we fix $\gamma$, the positive lower and upper bounds on $f$ and $f_r$ are also fixed independently of $r$. In particular, the solutions $u_r$ have uniform modulus of strict convexity (see \cite{C2}).
\end{rmk}

\subsection{Nonlinear Dependence and Conclusion of the Proof}
\label{sec:Nonlin Dependence}

Given $K>1$, let 
\[ u_{r,K}(x) := \frac{1}{Kr^{1/2}}u_r(A_{Kr^{1/2}}x). \]
By the homogeneity \eqref{eqn:scaling} and \eqref{eqn:ABP}, if $Kr^{1/2} \leq 1$, then 
\[\sup_{x_2 \in I} |u_{r,K}(0,x_{2}) - u(0,x_{2})| \leq \frac{1}{Kr^{1/2}}\sup_{x_2 \in I} |u_{r}(0,x_{2}) - u(0,x_{2})| \leq \frac{C}{K}. \]
Hence, since $u(0,\cdot)$ is homogeneous of degree $\frac{\gamma + 1}{\gamma} < 2$, taking $K$ large enough (the largeness depending only on $\gamma$)
implies that
\[ \osc_{I/2} \partial_{22} u_{r,K}(0,\cdot) \geq c > 0, \]
and so, by scaling, we conclude that
\[ \big[\partial_{22}u_{r}(0,\cdot)\big]_{C^{\alpha}(I/2)} \geq  K^{-\frac{\gamma - 1 + \alpha\gamma}{\gamma + 1}}r^{-\frac{\gamma - 1 + \alpha\gamma}{2(\gamma + 1)}}
\big[\partial_{22}u_{r,K}(0,\cdot)\big]_{C^{\alpha}(I/2)} \geq cr^{-\frac{\gamma - 1 + \alpha\gamma}{2(\gamma + 1)}}. \]
Thus,
\begin{equation}
\label{eqn:SecondDerivEstimate} 
\|D^2u_r\|_{C^{\alpha}(B_{1/2})} \geq cr^{-\frac{\gamma-1 + \alpha\gamma}{2(\gamma + 1)}}.
\end{equation}
Recalling \eqref{eqn:RHSCalpha}, we then see that
\begin{equation*}
\label{eqn:FinalInequality}
\|D^2u_r\|_{C^{\alpha}(B_{1/2})} \geq c\bigg(\frac{\|f_r\|_{C^{\alpha}(B_1)}}{C}\bigg)^{\frac{1}{2}\left(1 + \frac{\gamma - 1}{\alpha \gamma}\right)}.
\end{equation*}
In particular, given any $\alpha \in (0,1)$, we have superlinear dependence on $\|f_r\|_{C^{\alpha}(B_{1})}$ by choosing $\gamma > \frac{1}{1-\alpha}$. 

Finally, fix $\gamma = 2$, for example. Taking $\alpha = c/|\log r|$ in \eqref{eqn:RHSCalpha}, we have that $\|f_r\|_{C^{\alpha}(B_1)} \leq C$. Furthermore, by \eqref{eqn:SecondDerivEstimate} we see that
$\|D^2u_{r}\|_{C^{\alpha}(B_{1/2})}$ grows faster than a negative power of $r$, independent of $\alpha$. Therefore, we have at least exponential dependence on $1/\alpha$.

\section{Appendix}
\label{sec:appendix}

The goal of this appendix is to stress that the degeneracy and affine invariance of the Monge-Amp\`{e}re equation\,---\,and not its nonlinearity\,---\,forces $\omega$ to depend in a nonlinear fashion on $1/\alpha$ as $\alpha \to 0$ and $\|f\|_{C^{\alpha}(B_1)}$.

To show this, 
let us consider a fully nonlinear uniformly elliptic equation comparable to the Monge-Amp\`{e}re equation.
In particular, let 
$\mathcal{S}^n$ denote the space of real, symmetric $n \times n$ matrices, let
$F : \mathcal{S}^n \to \mathbb{R}$ be a uniformly elliptic operator with ellipticity constants $0 < \lambda \leq \Lambda < \infty$, and assume that $u$ is a solution of the equation
\begin{equation}
\label{eqn:u F}
F(D^2u(x)) = f(x) \qquad\text{in } B_1 \subset \mathbb{R}^{n}
\end{equation}
where $f \in C^{\alpha}(B_1)$.
Without loss of generality, we let $F(0) = f(0) = 0$.
The key assumption we make to place ourselves in a comparable setting to the Monge-Amp\`{e}re equation is to ensure that our operator $F$ has basic existence and regularity properties: there exist constants $\tilde{C} > 0$ and $\tilde \alpha \in (0,1]$ such that for any $M \in \mathcal{S}^n$ with $F(M) = 0$ and for every $\tilde{w} \in C(\partial B_1)$, there exists a unique solution $w$ to the equation
\begin{equation}
\label{eqn:assump1}
\begin{cases}
F(D^{2}w(x) + M) = 0 &\text{in } B_1 \\
w = \tilde{w} &\text{on } \partial B_1
\end{cases}
\end{equation}
and
\begin{equation}
\label{eqn:assump2}
\|w\|_{C^{2,\tilde \alpha}(B_{1/2})} \leq \tilde{C}\|w\|_{L^{\infty}(B_1)}.
\end{equation}
Indeed, for the Monge-Amp\`{e}re equation, \eqref{eqn:assump1} holds more appropriately replacing the right-hand side $0$ with the natural right-hand side $1$, $\tilde w$ with an affine function, and \eqref{eqn:assump2} holds with $\tilde \alpha=1$ thanks to \cite{P}. 

We remark that, by Evans-Krylov (see \cite{E,K}), 
if $F$ is concave on $\mathcal S^n$,
then the above assumptions are satisfied with some $\tilde\alpha$ depending on $\lambda,\Lambda,$ and $n$.

Analogous to what we have done in Section~\ref{sec:Pos Res 2}, 
since we want to understand the behavior of our estimates in $\alpha$ when $\alpha \to 0$,
we fix $\eta>0$ and we assume that $\alpha \leq \tilde \alpha-\eta$.
Hence, here we call ``universal'' any constant depending only on $\lambda, \Lambda, n,\tilde{C},\tilde \alpha$, and $\eta$.

Using \eqref{eqn:assump1} and \eqref{eqn:assump2}, if $\|u\|_{L^{\infty}(B_1)} \leq 1$ and $\|f\|_{L^{\infty}(B_1)}$ is sufficiently small, we can find a quadratic polynomial $\hat{Q}$ with $F(D^2\hat{Q})= 0$ that approximates $u$ at order $\hat{r}^{2+\alpha}$ in $B_{\hat{r}}$ for some small, universal $\hat{r} > 0$.
Then, if in addition $\|f\|_{C^{\alpha}(B_1)}$ is small, one can iterate this statement as in Lemma~\ref{lem:iteration} to produce the second order Taylor polynomial of $u$ centered at the origin and show that its Hessian is of order $1/\alpha$. 
(This follows from a careful reading of the $C^{2,\alpha}$ regularity result in \cite{C1}.) 
With this one deduces the following estimate:
\begin{thm}
\label{thm:CaffC2,a}
Fix $\eta>0$ and let $\alpha \in (0,\tilde \alpha -\eta].$
There exist positive universal constants $\bar{\epsilon}$ and $\bar C$ such that if
$u:B_1\to \mathbb R$ is a solution of \eqref{eqn:u F} with
\begin{equation}
\label{eqn:assump3}
\|u\|_{L^{\infty}(B_1)} \leq 1 
\qquad\text{and}\qquad
\|f\|_{C^{\alpha}(B_1)} \leq \bar{\epsilon},
\end{equation}
then
$$
\|u\|_{C^{2,\alpha}(B_{1/2})} \leq \frac{\bar C}{\alpha}.
$$
\end{thm}

To remove the smallness assumptions (\eqref{eqn:assump3}) on $u$ and $f$, given a general solution $u$ of \eqref{eqn:u F}, we define
$$
\epsilon := \frac{\bar{\epsilon}}{\|u\|_{L^{\infty}(B_1)} + \|f\|_{C^{\alpha}(B_1)}},\qquad v := \epsilon u,\qquad g := \epsilon f,
\qquad\text{and}\qquad G(M) := \epsilon F(M/\epsilon).
$$
As defined, $G$ is a uniformly elliptic operator with ellipticity constants $\lambda \leq \Lambda$ such that $G(0) = 0$ and satisfies \eqref{eqn:assump1} and \eqref{eqn:assump2}. 
Furthermore, $G(D^2v) = g$, $g(0) = 0$, $\|v\|_{L^{\infty}(B_1)} \leq 1$, and $\|g\|_{C^{\alpha}(B_1)} \leq \bar{\epsilon}$. 
Hence, applying Theorem~\ref{thm:CaffC2,a} to $v,g$, and $G$, we obtain
$$
\|v\|_{C^{2,\alpha}(B_{1/2})} \leq \frac{\bar{C}}{\alpha},
$$
or equivalently 
\[
\|u\|_{C^{2,\alpha}(B_{1/2})} \leq \frac{\bar{C}}{\alpha\bar{\epsilon}}\big(\|u\|_{L^{\infty}(B_1)} + \|f\|_{C^{\alpha}(B_1)}\big),
\]
as desired.

Notice that, in contrast, for the Monge-Amp\`ere equation, the key step Lemma~\ref{lem:basic} requires $u$ to be close to a ``round'' quadratic polynomial due to the degeneracy of the equation. To make sure this holds through the iteration we must start $\alpha$-close to such a quadratic polynomial. In turn, to guarantee this we must rescale by a (possibly very eccentric) affine function, which explains the nonlinear dependence.


\section*{Acknowledgments}
A. Figalli was supported by NSF grant DMS-1262411 and NSF grant DMS-1361122.
C. Mooney was supported by NSF grant DMS-1501152. C. Mooney would like to thank Simon Brendle, Tianling Jin and Luis Silvestre for encouragement and helpful conversations.



\end{document}